\documentclass[a4paper,reqno]{amsart}
\usepackage[a4paper]{geometry}
\usepackage{mathtools,amsfonts,amsthm,esint}

\newcommand{\R}{\mathbb{R}}
\newcommand{\N}{\mathbb{N}}
\newcommand{\Z}{\mathbb{Z}}
\newcommand{\Ha}{\mathcal{H}}
\newcommand{\e}{\varepsilon}
\def\LM#1{\hbox{\vrule width.2pt \vbox to#1pt{\vfill \hrule width#1pt
height.2pt}}}
\def\LL{{\mathchoice {\>\LM7\>}{\>\LM7\>}{\,\LM5\,}{\,\LM{3.35}\,}}}
\def\mres{{\LL}}

\usepackage[dvipsnames]{xcolor}

\newcommand{\ol}{\overline}
\newcommand{\Sd}{\mathbb{S}^{d-1}}
\newcommand{\hn}{\mathcal{H}^{d-1}}
\newcommand{\Mnn}{{\mathbb{M}^{d\times d}_{sym}}}
\newcommand{\uxz}{u^\xi_z}
\newcommand{\xz}{^\xi_z}

\newtheorem{theorem}{Theorem}
\newtheorem{lemma}[theorem]{Lemma}
\newtheorem{prop}[theorem]{Proposition}
\newtheorem{corollary}{Corollary}
\theoremstyle{remark}
\newtheorem{remark}{Remark}

\title[A characterization of $G(S)BD$]{A characterization of Generalized \\ Functions of Bounded Deformation}
\author{Antonin Chambolle}
\address{CEREMADE, CNRS, Université Paris-Dauphine, Univ.~PSL, France
  and Mokaplan, INRIA Paris.}
\email[Antonin Chambolle]{antonin.chambolle@ceremade.dauphine.fr}
\author{Vito Crismale}
\address{Dipartimento di Matematica Guido Castelnuovo, Piazzale Aldo Moro 5, 00185 Roma, Italy.}
\email[Vito Crismale]{vito.crismale@mat.uniroma1.it}

\begin{document}
\maketitle


\begin{abstract}
  We show that Dal~Maso's $GBD$ space, introduced for tackling crack
  growth in linearized elasticity, can be defined by simple conditions
  in a finite number of directions of slicing.
\end{abstract}
\section{Introduction}

In \cite{DM13} G.\ Dal Maso introduced the space of \textit{Generalized (Special) functions of Bounded Deformation} ($G(S)BD$) to characterize the domain of a class of functionals associated to \textit{free discontinuity problems}~\cite{DGA,AFP} in linearized elasticity or elastoplasticity, in particular for the modeling of crack growth~\cite{FM}. In these contexts the energy functionals depend on a variable $u\colon \Omega{\subset} \R^d \to \R^d$, representing the \emph{displacement}, through its symmetrized gradient $e(u)=(Du{+}Du^T)/2$ (the \emph{deformation}) and its $(d{-}1)$-dimensional jump set $J_u$ (the \emph{crack set}), with a constraint imposing - in a weak sense - Dirichlet boundary conditions. 
Yet an integrability bound on $u$ in terms of a bound for such functionals is in general not available, so that it is not possible to control the deformation in the space of Radon measures.

An analogous issue arises both in the simplified \emph{antiplane shear} setting, where the displacement is only in the vertical direction so it can be assumed scalar, and in the framework of finite elasticity or elastoplasticity, contexts in which the role of $e(u)$ is played by the full gradient $\nabla u$. This motivated the introduction of the space $G(S)BV$~\cite{Amb90GSBV, Amb94} as the space of functions which are in $(S)BV$ after composition $\psi\circ u$
with $C^1$ functions $\psi$ having $\nabla \psi$ compactly supported.
 
However, $e(\psi\circ u)$ depends on the whole 
gradient of $u$, even for $u$ smooth,
which prevents from adopting a similar definition of $G(S)BD$.
The solution found in~\cite{DM13} was to define this space
 by slicing: a measurable function $u:\Omega\to\R^d$ is in $GBD(\Omega)$ if there exists a bounded Radon measure $\Lambda_u$ such that for every $\xi \in \Sd$ and for $\Ha^{d-1}$-a.e.\ $z\in \xi^\perp$ the slice  $s\mapsto \uxz(s):=\xi\cdot u(z+s\xi)$
    has bounded variation in $\Omega\xz:=\{s\in\R: z+s\xi\in\Omega\}$
    and
    \[
     (\widehat{\mu}_u)^\xi(B):= \int_{\xi^\perp} \Big((|D^a \uxz |+|D^c \uxz|)(B\xz) + \sum_{s\in J_{\uxz}\cap B\xz}|[\uxz](s)|\wedge 1\Big) d\Ha^{d-1}(z) \le \Lambda_u(B)
    \]
    for every $B\subset \Omega$ Borel, with $B\xz:=\{s\in\R: z+s\xi\in B\}$.
    Still in~\cite{DM13}, it is shown that it is enough to have the bound for
      all $\xi$ in a dense subset of $\mathbb{S}^{d-1}$. However,
      it is well known that a much simpler characterization exists for
      the space $BD$ of functions of \emph{Bounded Deformation}:
      a measurable $u$ is in $BD(\Omega)$ as soon as there is a basis $(e_i)_{i=1}^d$
      of $\R^d$ such as the following holds:
      for any $\xi \in \tilde{V} :=  \{ e_i :{i=1,\dots,d}\} \cup\{ e_i+e_j :{1\le i<j\le d}\}$ and for $\Ha^{d-1}$-a.e.~$z\in \xi^\perp$,
 $\uxz\in BV(\Omega\xz)$ 
    and
    \[
    |\mathrm{E}u \xi\cdot \xi| (\Omega)= |\mathrm{D}_\xi (u\cdot \xi)|(\Omega)=\int_{\xi^\perp} |D \uxz |(\Omega\xz) d\Ha^{d-1}(z) <+\infty,
    \]
    \textit{cf.}~\cite[Section~3]{AmbCosDM97} and \cite[Chap. II, Section 2.2]{Tem} (see also \cite{TemStr, Tem, Bab15, DePRin16, DePRin20} for general properties of $BD$ functions).
    It is therefore natural to ask whether $GBD$ functions can, similarly,
    be described by conditions in a finite number of appropriate directions.
    Our main result answers this affirmatively.

We show that a measurable $u$ is in $GBD(\Omega)$ as soon as there exists a basis $(e_i)_{i=1}^d$ of $\R^d$ and $d(d-1)/2$ vectors
    $\xi_{i,j}\in \text{span}\{e_i\pm e_j\}$, $1{\le} i{<}j{\le} d$ such that
    for any 
 $\xi \in V =  \{e_i:{i=1,\dots,d}\} \cup \{\xi_{i,j}\colon {1\le i<j\le d}\}$
    and
    for $\Ha^{d-1}$-a.e.~$z\in \xi^\perp$, $ \uxz \in BV(\Omega\xz)$
    and
    \[
     \Lambda^\xi_u:=(\widehat{\mu}_u)^\xi(\Omega) <+\infty.
    \]
    In other words, we show that if we have a control on the
    $d(d+1)/2$ 
    directions $\xi \in V$ through $\Lambda^\xi_u$, then the same control
holds in fact for \textit{all} $\xi\in\mathbb{S}^{d-1}$, and can
be localized on Borel subsets  $B\subset \Omega$ and controlled by  a global
finite measure $\Lambda_u$.

Our strategy is to approximate any measurable $u$ with $\Lambda_u^\xi$ bounded for all $\xi \in V$, by a family of functions, each in a suitable finite dimensional space, satisfying a uniform bound in $GBD$, according to the definition in
\cite{DM13} (Theorem~\ref{thm:main}).
More precisely, for any $\e>0$, we choose a family of
discretizations of the domain into hypercubes of sidelength $\e>0$ so that $u$ is approximated in measure by the functions defined in any hypercube as the $d$-linear interpolation over the vertices (see Appendix~\ref{sec:disc}). Then we modify such functions, setting them to 0 on the  ``bad hypercubes'', which are such that there are $\xi \in V$ and a couple of vertices joined by a segment parallel to $\xi$ that intersect the set where the jump of $u$ is larger than 1.
Thanks to an averaging argument, we show that we can choose
  the discretizations in such a way that (i) the bound on $\Lambda_u$ gives a uniform control on the perimeter of the union of the ``bad hypercubes'' while (ii) outside this set, the symmetric gradient of the $d$-linear approximation is controlled in $L^1$ by the finite differences along directions in $V$, thus again by the bound on $\Lambda_u$ (see Appendix~\ref{sec:energy}).

{}From this approximation we derive (Corollary~\ref{cor:GBDnew}) that $u\in GBD$
by employing a suitable compactness result for bounded sequences in $GBD$ that converge in measure \cite{DM13, Fri19}. This strategy of discretization/reinterpolation for $(G)BD$ functions was first used
in~\cite{Chambolle-SBD2D,Chambolle-SBDND,Iur14}
to show approximation results in $(G)SBD$.

As a further application (Corollary~\ref{cor:comp}), we prove that any sequence of measurable functions satisfying uniformly the control on $(\Lambda_{u_k}^\xi)_{\xi \in V_k}$ is precompact in $GBD$, in the sense of the general compactness and closedness result in \cite{ChaCriIUMJ} (we refer also to 
\cite{AlmiTasso} 
for a related compactness result in $GBD$).

Eventually, we address an open question raised in the recent work \cite{ADKT} (see also \cite{AlmiTasso2}) concerning a nonlocal approximation of Griffith-type energies for brittle cracks (see e.g.\ \cite{FM, CCF16, FriPWKorn, 
FriSol16, CC, CFI19,CCI19,CC-CalcVar, CagChaSca20, FriPerSol21, BabIurLem22, FriLabSti23}). In some approximation regimes any limiting admissible deformation $u$ naturally belong to the space therein called $GBV^{\mathcal{E}}(\Omega; \R^d)$, namely it satisfies the bound 
\[
\int_{\Sd} \widehat{\mu}_u^\xi(B) \,d\Ha^{d-1}(\xi) \leq \Lambda_u(B) \quad \text{for } B\subset \Omega \text{ Borel,}
\]
for a suitable bounded Radon measure $\Lambda_u$. We prove that $GBV^{\mathcal{E}}$ coincide with $GBD$ (Corollary~\ref{cor:GBVE}), in particular we give positive answer to a question in \cite{ADKT} concerning the rectifiability of the measure obtained by integrating over directions $\xi$ the jump parts of the measures $(\widehat{\mu}_u)^{\xi}$.

In the same spirit, we prove  (Proposition~\ref{prop:mainp}) a characterization of the space $GSBD^p$ (i.e.\ the $GBD$ functions with slices in $SBV$, symmetrized gradient in $L^p$ and jump set of finite $\Ha^{d-1}$-measure, see e.g.\ \cite{FriPWKorn}), with a slight adaptation of the arguments in the proof of Theorem~\ref{thm:main}, and the corresponding
compactness (Corollary~\ref{cor:GSBDp}), employing the general compactness result \cite{ChaCri20AnnSNS}, see also \cite{CC-JEMS}.
\section{Main results}

Here $\Omega$ is a bounded\footnote{for simplicity, since the result could
  easily be localized.} open set of $\R^d$, $d\ge 2$.
Given $u:\Omega\to\R^d$, for any $\xi\in\R^d\setminus \{0\}$ and
$z\in\xi^\perp$, we define:
\[
\uxz(s) := \xi\cdot u(z+s\xi) \quad \text{for }s\in \Omega\xz:= \{s\in\R: z+s\xi\in\Omega\}.
\] 
We denote by $L^0(\Omega; \R^m)$, $m\ge 1$, the measurable functions from $\Omega$ into $\R^m$ and by $v_k\to v$ in $L^0(\Omega; \R^m)$ we mean that $v_k$ converge to $v$ in $\mathcal{L}^d$-measure in $\Omega$, for $\mathcal{L}^d$ the Lebesgue measure on $\R^d$. 
We also let $Q:=[0,1)^d$  and, for any $\e>0$, 
\[
\Omega_\e:=\{x \in \Omega\colon \mathrm{dist}(x, \partial^* \Omega) > \sqrt{d}\e\}.
\] 
We refer to \cite{AFP} for general theory of $BV$ functions. If $u \in BV$ we denote by $D^a u$ and $D^c u$ the absolutely continuous and the Cantor part of the bounded Radon measure $Du$. Moreover, we refer to \cite{AmbCosDM97} for a careful treatment of the $SBD$ space and to \cite{DM13} for the definition and the main properties of $GBD$ functions.

\begin{theorem}\label{thm:main}
  Let $u \in L^0(\Omega; \R^d)$
  and assume
  there is an orthonormal basis $(e_i)_{i=1}^d$ of $\R^d$ such that for any $\e>0$ it holds that
  $\xi\in V :=  \{e_i:{i=1,\dots,d}\} \cup \{e_i+e_j\colon {1\le i<j\le d}\}$, and that for
  $\Ha^{d-1}$-a.e.\ $z\in\xi^\perp$, $\uxz\in BV(\Omega\xz)$
  and
  \begin{equation}\label{Lambdauxi}
    \Lambda_u^\xi :=
     \int_{\xi^\perp} \Bigg((|D^a \uxz|+|D^c \uxz|)({\Omega\xz}) + \sum_{s\in J_{\uxz}}
     |[\uxz(x)]|\wedge 1\Bigg) d\Ha^{d-1}(z) <+\infty.
  \end{equation}
  Then there exist a  dimensional constant $C>0$ and a family $(u^\e)_{\e>0}\subset SBV^\infty(\Omega;\R^d)$ 
  such that, for any $\e>0$,
   $J_{u_\e}\subset \partial \mathcal{B}^\e$ with $\mathcal{B}^\e$ union of finitely many cubes of sidelength $\e$ included in $\Omega$, 
 \begin{equation}\label{boundunifapp}
    \int_{ \Omega_\e } | e(u^\e)| dx +  \Ha^{d-1}(\partial^* \mathcal{B}^\e  \cap \Omega_\e ) 
     \le C\,  \Lambda^V_u, \quad \Lambda^V_u:= \sum_{\xi \in V}\Lambda_u^\xi,
  \end{equation} 
and $u^\e \to  u\in L^0(\Omega; \R^d)$.
\end{theorem}
\begin{remark}\label{rem:nonorth}
    In Theorem~\ref{thm:main}, we assumed $(e_i)_{i=1}^d$ orthonormal for simplicity.
    The result would hold also for an arbitrary basis:
    this is easily checked by considering a matrix $A$ which
    sends $(e_i)_{i=1}^d$ to an orthonormal basis and replacing $u$
    with $A^{-T}u(A^{-1}\cdot)$.
    One could of course also consider
    the integrals in~\eqref{Lambdauxi} over arbitrary spaces $E_\xi$ with
    $E_\xi\oplus \R\xi = \R^d$ rather than just $\xi^\perp$, see \cite[Remark~4.11]{DM13}. Additionally,
    as in~\cite[Remark~4.3]{DM13},
    one could replace in~\eqref{Lambdauxi} the threshold $1$ (in $|[\uxz]|\wedge 1$)
    by an arbitrary ($\xi$-dependent) threshold level $\beta_\xi>0$,
    since $|x|\wedge \alpha \le (1\vee \alpha/\beta)(|x|\wedge \beta)$
    for any $\alpha,\beta>0$.
    Finally one could consider, for some pairs $(i,j)$, slicing in the
    direction $e_i-e_j$ rather than $e_i+e_j$.
  \end{remark}

From Theorem~\ref{thm:main} we deduce the following three consequences. We refer to \cite[Definition~2.3]{ADKT} for the definition of the space $GBV^{\mathcal{E}}(\Omega; \R^d)$, characterized in Corollary~\ref{cor:GBVE}.

\begin{corollary}\label{cor:GBDnew}
For any fixed basis $(e_i)_{i=1}^d$ of $\R^d$ and  $ V :=  \{e_i:{i=1,\dots,d}\} \cup \{e_i+e_j\colon {1\le i<j\le d}\}$ it  holds that
\begin{equation}\label{GBDequiv}
GBD(\Omega)=\{u \in L^0(\Omega; \R^d) \colon \Lambda^V_u < +\infty\},
\end{equation}
where, for any $u \in L^0(\Omega; \R^d)$, $\Lambda^V_u$ is defined  in \eqref{Lambdauxi}, \eqref{boundunifapp}.
\end{corollary}

\begin{corollary}\label{cor:comp}
Let $(u_k)_k \subset GBD(\Omega)$ such that for every $k \in \N \setminus \{0\}$ there exists $(e^k)_k$, $e^k:=(e^k_i)_{i=1}^d$ orthonormal basis of $\R^d$ with (recalling \eqref{Lambdauxi}, \eqref{boundunifapp}) 
\begin{equation*}
\sup_{k\in \N} \Lambda^{V_k}_{u_k}<+\infty,
\end{equation*}
for $V_k:=\{e^k_i\colon i=1,\dots,d\} \cup \{e^k_i+e^k_j\colon 1\le i<j\le d\}$. Then
there exist 
a Caccioppoli partition $\mathcal{P}=(P_n)_n$ of $\Omega$, a sequence of piecewise rigid motions   $(a_k)_k$ with
 \begin{subequations}\label{eqs:0808222034}
\begin{equation}\label{2302202219}
a_k=\sum_{n\in \N} a_k^n \chi_{P_n},
\end{equation}
\begin{equation}\label{2202201909}
|a_k^n(x)-a_k^{n'}(x)| \to +\infty \quad \text{for a.e.\ }x\in \Omega, \text{ for all }n\neq {n'},
\end{equation}
\end{subequations}
 and $u\in GBD(\Omega)$ such that, up to a (not relabelled) subsequence, 
\begin{subequations}\label{eqs:0203200917}
\begin{align}
u_k-a_k &\to u \quad \text{a.e.\ in }\Omega, \label{2202201910}\\
\Ha^{d-1}(\partial^* \mathcal{P} \cap \Omega)&\leq \lim_{\sigma\to +\infty}\liminf_{k\to \infty} \,\Ha^{d-1}(J^\sigma_{u_k}).\label{eq:sciSaltoInf}
\end{align}
\end{subequations}
where $J^\sigma_{u_k}:=\{x \in J_{u_k}\colon |[u_k]|(x)\geq \sigma \}$. 
\end{corollary}
\begin{corollary}\label{cor:GBVE}
It holds that
\begin{equation*}
GBV^{\mathcal{E}}(\Omega; \R^d)=GBD(\Omega).
\end{equation*}
\end{corollary}

We now provide analogous results for the space $GSBD^p$ of $GBD$ functions with symmetrized gradient in $L^p$ and jump set of finite $\Ha^{d-1}$-measure, see e.g.\ \cite{CC}.
\begin{prop}\label{prop:mainp}
  Let $u \in L^0(\Omega; \R^d)$, $p>1$,
  and assume
  there is an orthonormal basis $(e_i)_{i=1}^d$ of $\R^d$ such that for any $\e>0$ it holds that
  $\xi\in V :=\{e_i\colon i=1,\dots,d\} \cup \{e_i+e_j\colon 1\le i<j\le d\}$, and
  almost $z\in\xi^\perp$, $\uxz\in SBV^p(\Omega\xz)$ with $D^a \uxz=(\uxz)'\mathcal{L}^1$
  and
  \begin{equation}\label{Lambdauxip}
    \Lambda_u^{p,\xi} :=
     \int_{\xi^\perp} \Big( \int_{\Omega\xz}|(\uxz)'|^p \, dt + \#J_{\uxz} \Big) d\Ha^{d-1}(z) <+\infty.
  \end{equation}
  Then there exist a  constant $C>0$ depending only on $d$, $p$, and a family $(u^\e)_{\e>0}\subset SBV^\infty(\Omega;\R^d)$ 
  such that, for any $\e>0$, $J_{u_\e}\subset \partial \mathcal{B}^\e$ with $\mathcal{B}^\e$ union of finitely many cubes of sidelength $\e$ included in $\Omega$,
 \begin{equation}\label{boundunifappp}
    \int_{ \Omega_\e } | e(u^\e)|^p dx +  \Ha^{d-1}(\partial^* \mathcal{B}^\e  \cap \Omega_\e ) 
     \le C\,  \Lambda_u^{p,V}, \quad \Lambda^{p,V}_u:= \sum_{\xi \in V}\Lambda_u^{p,\xi},
  \end{equation} 
and $u^\e \to  u\in L^0(\Omega; \R^d)$.
In particular $u \in GSBD^p(\Omega)$.
\end{prop}
\begin{corollary}\label{cor:GSBDp}
With the notation of Proposition~\ref{prop:mainp}, let $(u_k)_k \subset GSBD^p(\Omega)$, $(e^k)_k$, $(V_k)_k$ with $e^k:=(e^k_i)_{i=1}^d$ orthonormal basis of $\R^d$, $V_k:=\{e^k_i \colon i=1,\dots,d\} \cup \{e^k_i+e^k_j\colon 1\le i<j\le d\}$ be sequences such that 
\[
\sup_{k\in \N} \Lambda^{p,V_k}_{u_k}<+\infty.
\] Then there exist a Caccioppoli partition $\mathcal{P}=(P_n)_n$ of $\Omega$, a sequence of piecewise  rigid motions 
  $(a_k)_k$ with
  \begin{subequations}
\begin{equation}\label{2302202219'}
a_k=\sum_{n\in \N} a_k^n \chi_{P_n}\,,
\end{equation}
\begin{equation}\label{2202201909'}
|a_k^n(x)-a_k^{n'}(x)| \to +\infty \quad \text{for a.e.\ }x\in \Omega, \text{ for all }n\neq n'\,,
\end{equation}
\end{subequations}
and $u\in GSBD^p(\Omega)$ such that, up to a (not relabelled) subsequence, 
\begin{subequations}
\begin{align}
u_k-a_k \to u \quad &\text{a.e.\ in }\Omega\,, \label{2202201910'}\\
e(u_k)  \rightharpoonup e(u) \quad &\text{in } L^p(\Omega; \Mnn)\,,\label{eq:convGradSym}\\
\hn(J_u \cup \partial^* \mathcal{P} )\leq \liminf_{k\to \infty} \,&\hn(J_{u_k})\,.\label{eq:sciSalto}
\end{align}
\end{subequations}
\end{corollary}
From now on we prove the announced results, following the order of presentation.
\begin{proof}[Proof of Theorem~\ref{thm:main}]
   For
   $\xi\in V$ and $z\in \xi^\perp$ such that $\uxz\in BV(\Omega\xz)$,
   we denote by $\mu\xz$ the measure (on $\Omega\xz$):
   \[
     \mu\xz =  |D^a \uxz| + |D^c \uxz| + \sum_{s\in J_{\uxz}}
     (|[\uxz](s)|\wedge 1)\Ha^0\mres \{s\}.
   \]
Fixed $\e>0$, we evaluate the integral
  \[
    \e^{d-1}  \int_{Q} \sum_{\xi\in V}\sum_{i\in\e\Z^d}\big(|\xi\cdot(u(\e y + i + \e\xi)-u(\e y+ i))|\wedge 1 \big) dy
  \]
  where here the sum is over $i\in\e\Z^d$ such that
  $\e y+ i\in \Omega\cap (\Omega-\e\xi)$.
  Given $\xi\in V$, then the change of variable $x=i+\e y$ shows that:
  \begin{multline*}
    \e^{d-1}  \int_{Q}\sum_{i\in\e\Z^d} \big(|\xi\cdot(u(\e y + i + \e\xi)-u(\e y+ i))|\wedge 1 \big) dy
    \\=
    \e^{-1}  \int_{\Omega\cap (\Omega-\e\xi)} \big(|\xi\cdot(u(x+\e\xi)-u(x))|\wedge 1 \big)dx,
   \end{multline*}
   and writing $x=z+s\xi$, $z\in\xi^\perp$, $s\in\Omega\xz\cap (\Omega\xz-\e)$
   we find that this is also expressed as
   \[
     \e^{-1}\int_{\xi^\perp}\int_{\Omega\xz\cap (\Omega\xz-\e)}  \big(
     |\uxz(s+\e)-\uxz(s)|\wedge 1 \big)|\xi|ds \,d\Ha^{d-1}(z).
   \]
   Now, since (for $\Ha^{d-1}$-a.e.~$z$) $\uxz\in BV(\Omega\xz)$, then
   for a.e.~$s$:
   either there is $s'\in (s,s+\e)\cap J_{\uxz}$  with $|[\uxz](s)|\ge 1$,
   and
   then clearly $|\uxz(s+\e)-\uxz(s)|\wedge 1\le\mu\xz((s,s+\e))$,
   or there isn't, in which case,
   $\mu\xz\mres (s,s+\e)=|D\uxz|\mres (s,s+\e)$, which implies
   that $|\uxz(s+\e)-\uxz(s)|\le \mu\xz((s,s+\e))$.
   It follows that
   \begin{multline*}
     \int_{\Omega\xz\cap (\Omega\xz-\e)} 
     \Big(|\uxz(s+\e)-\uxz(s)|\wedge 1\Big)|\xi|\,ds
     \\ \le |\xi| \int_{\Omega\xz\cap (\Omega\xz-\e)} 
     \int_{}  \chi_{(s,s+\e)}(t) d\mu\xz(t) \, ds 
     \le \e|\xi|\int_{\Omega\xz} d\mu\xz(t)
   \end{multline*}
   where we have used Fubini's theorem and that $\chi_{(s,s+\e)}(t)=
   \chi_{(t-\e,t)}(s)$.
   Combining the previous estimates, we end up with
   \begin{multline*}
     \e^{d-1}  \int_{Q} \sum_{\xi\in V}\sum_{i\in\e\Z^d}\big(|\xi\cdot(u(\e y + i + \e\xi)-u(\e y+ i))|\wedge 1 \big) dy
     \\ \le \sum_{\xi\in V}|\xi|\int_{\xi^\perp}\int_{\Omega\xz} d\mu\xz d\Ha^{d-1}(z) \le \sum_{\xi\in V}|\xi|\Lambda_u^\xi 
     =:M.
   \end{multline*}
   Hence, there is a set $Q^\e\subset Q$ of measure at least $1/2$ such
   that for any $y\in Q^\e$:
   \begin{equation}
          \e^{d-1} \sum_{\xi\in V}\sum_{i\in\e\Z^d}\big(|\xi\cdot(u(\e y + i + \e\xi)-u(\e y+ i))|\wedge 1 \big) \le 2M.
   \end{equation}
  Let us choose $y^\e \in Q^\e \cap Q_\e$, where $Q_\e\subset Q$ is the set given by Proposition~\ref{prop:disc} in Appendix~\ref{sec:disc}. 
 For $\e>0$ small enough  and for any $i \in \e \Z^d$ we denote by
 \[
 Q_\e^i:=\e y^\e+i+\e [0,1)^d
 \] the cube corresponding to $i$,  by
\[
   \mathcal{V}_\xi(Q_\e^i):=\big\{j \in \e\Z^d \colon \e y^\e+j+[0,\e)\xi \text{ is an edge of } Q_\e^i \big\}
\] 
the vertexes of $Q_\e^i$ which are endpoints of edges parallel to $\xi$ (more precisely, we consider for any such edge the lowest endpoint, with respect to the order in $\Z^d$, that is the one included in $Q_\e^i$), and 
  we define the function $u^\e$ for every $x \in \Omega$ as
   \begin{itemize}
   \item the multilinear interpolation of the values  of $u$ 
   at the vertices
   of the cube  $Q_\e^i$ 
   which contains $x$, if  $\ol Q_\e^i \subset \Omega$ 
   and if 
   for all $\xi\in V$ and  $j \in \mathcal{V}_\xi(Q_\e^i)$ 
     \[
       |\xi\cdot (u(\e y^\e + j + \e\xi)-u(\e y^\e + j) )|\leq 1 
       ;
     \] 
   \item $0$, else.
   \end{itemize}
   Observe that  $u^\e$ is a polynomial in the cubes of the first type, so in particular it does not jump therein, that it is continuous accross a facet between two such cubes, and that the number of cubes of the second type
   is bounded by 
$(2/\e)^{d-1} 2M$
    so that,  denoting by $\mathcal{B}^\e$ the union of cubes of the second type, it holds that 
   \begin{equation}\label{controllobadcubes}
   J_{u^\e}\subset \partial \mathcal{B}^\e, \quad \Ha^{d-1}(\partial \mathcal{B}^\e)\leq CM, \quad |\mathcal{B}^\e|\leq CM\e
 \end{equation}
 for  a dimensional $C>0$.
 In particular, the convergence properties of the multilinear interpolations
 stated in Proposition~\ref{prop:disc} also hold for $u^\e$, since they differ, in $\Omega$,
 on a set of vanishing measure.
 In addition, 
 if  $Q_\e^i$ 
   is a cube of the first type, 
we have that
\begin{equation*}
    \int_{Q_\e^i} |e(u^\e)| dx\leq C \e^{d-1}  \sum_{\xi\in V}\sum_{j\in \mathcal{V}_\xi(Q_\e^i)}|\xi\cdot(u(\e y^\e + j + \e\xi)-u(\e y^\e+ j))|, 
    \end{equation*}
    by Proposition~\ref{prop:control} (in Appendix~\ref{sec:energy}) for $p=1$, and a change of variables.
    
    Therefore,  summing the above two estimates over the cubes all included in $\Omega$
    (which cover $\Omega_\e$),
     we find
   that, for $\e$ small enough, $u^\e\in  SBV^\infty(\Omega)$ 
   with
 \[
     \int_{ \Omega_\e } | e(u^\e) | dx +  \Ha^{d-1}(J_{u^\e}\cap  \Omega_\e ) 
     \le 2CM,
   \]
for a  dimensional  constant $C>0$, and this gives \eqref{boundunifapp} recalling that $M=\sum_{\xi\in V}|\xi|\Lambda_u^\xi$.
\end{proof}

\begin{proof}[Proof of Corollary~\ref{cor:GBDnew}]
Let $(u^\e)_{\e>0}$ be the sequence provided by Theorem~\ref{thm:main}
 and fix $\Omega'\subset \subset \Omega$.  By the slicing properties of $BV$ functions and the bound \eqref{boundunifapp} it follows that defining  
$\Lambda^\e:=  |e(u^\e)|dx + \Ha^{d-1}\mres \partial \mathcal{B}^\e$, 
   then $\Lambda^\e( \Omega' )\le C\, \Lambda_u$ and
   one has, for all $\xi\in \mathbb{S}^{d-1}$ and all Borel $B\subset  \Omega' $,
\begin{equation}\label{stimagbdclassica}
     \int_{\xi^\perp} \bigg( \int_{B\xz\setminus J_{(u^\e)\xz}} | e((u^\e)\xz) | dt  +  \# 
        (J_{(u^\e)\xz}\cap B\xz)  \bigg) d\Ha^{d-1}(z)
     \le \Lambda^\e(B)
   \end{equation}
   so that $(u^\e)_\e$ are equibounded in  $GBD( \Omega' )$, in the sense of \cite{DM13}.
    

   Moreover, as $u^\e \to u$ in $L^0(\Omega; \R^d)$, by \cite[Theorem~1.1]{ChaCriIUMJ} we conclude that $u \in GBD( \Omega' )$. In fact, the  convergence of $u^\e$ to $u$ implies that the infinitesimal rigid motions can be taken as equal to 0 (alternatively, one could use \cite[Corollary 11.2 and Theorem~11.3]{DM13} together with the fact that since $u^\e\to u$ in $L^0(\Omega; \R^d)$, there is an increasing function $\psi_0\colon \R^+\to \R^+$ with $\lim_{s\to +\infty}\psi_0(s)=+\infty$ such that $\|\psi_0(u^\e)\|_{L^1(\Omega)}$ is uniformly bounded w.r.t.\ $\e>0$, see e.g.\ \cite[Lemma~2.1]{Fri19}).
 Then we conclude that $u \in GBD(\Omega)$ by the arbitrariness of $\Omega' \subset\subset\Omega$.  
  \medskip
\end{proof}
\begin{proof}[Proof of Corollary~\ref{cor:comp}] For any $u_k$ let $(u^\e_k)_{\e>0}$ be the approximating family provided by Theorem~\ref{thm:main}.
Then for any $k \in \N \setminus \{0\}$ we can find $\e_k>0$ such that, for $\tilde{u}_k:=u^{\e_k}_k$ it holds that $\tilde{u}_k \in SBV^\infty(\Omega; \R^d)$  with $J_{\tilde{u}_k} \subset \partial^* \mathcal{B}^{\e_k}$ and $\mathcal{B}^{\e_k}$ union of finitely many cubes of sidelength $\e_k$ included in $\Omega$, such that
\begin{equation}\label{0202252211}
\begin{split}
\|\arctan(\tilde{u}_k-u_k)\|_{L^1}&\leq \frac{1}{k},\\
\int_{ \Omega_{\e_k} } | e(\tilde{u}_k)| dx +  \Ha^{d-1}(\partial^* \mathcal{B}^{\e_k} \cap \Omega_{\e_k} ) 
     &\le C\,  \Lambda^{V_k}_{u_k}.
\end{split}
\end{equation}
Recalling \eqref{stimagbdclassica}, it follows that the sequence $(\tilde{u}_k)_k$ satisfies the assumptions of \cite[Theorem~1.1]{ChaCriIUMJ} 
on any $\Omega_{\e_{\overline k}}$ for $k\geq \overline{k}$,  and therefore by that compactness result  and by a diagonal argument on $\Omega_{\e_k}$,  up to a (not relabelled) subsequence, there are
a Caccioppoli partition $\mathcal{P}=(P_n)_n$ of $\Omega$, a sequence of piecewise rigid motions $(a_k)_k$ satisfying \eqref{eqs:0808222034}, and
$u\in GBD(\Omega)$ such that 
\begin{subequations}\label{eqs:0203200917'}
\begin{equation}\label{2202201910''}
\tilde{u}_k-a_k \to u \quad \text{a.e.\ in }\Omega. 
\end{equation}
\end{subequations}
By the first in \eqref{0202252211} we deduce \eqref{2202201910}.
At this stage, \eqref{eq:sciSaltoInf} follows arguing exactly as in \cite[Section~3.3: Lower semicontinuity]{ChaCriIUMJ}.
\end{proof}

%
%

\begin{proof}[Proof of Corollary~\ref{cor:GBVE}]
  We prove that $GBV^{\mathcal{E}}(\Omega; \R^d)\subset GBD(\Omega)$, the opposite inclusion being true by definition of $GBD(\Omega)$ in \cite{DM13}.
  The proof follows by averaging.
    We fix an orthonormal basis $(e_i)_{i=1}^d$, and let $\widehat{V} = \{e_i,i=1,\dots,d;
    \frac{e_i+e_j}{\sqrt{2}}, 1\le i<j\le d\}$.
    Then, we have that:
    \[
      \fint_{SO(d)}\sum_{\xi\in \widehat{V}} \Lambda_u^{R\xi} d\mu(R) =\frac{d(d+1)}{2}\fint_{\mathbb{S}^{d-1}} \Lambda_u^\xi d\Ha^{d-1}(\xi)
      \] (for $\mu$ the Haar measure over $SO(d)$), which is finite whenever
      $u\in GBV^{\mathcal{E}}(\Omega; \R^d)$. We then pick any rotation
      for which $\sum_{\xi\in R\,\widehat{V}}\Lambda_u^\xi$ is less than average and
      consider the basis $(Re_i)_{i=1}^d$ in Theorem~\ref{thm:main}. We conclude since $\Lambda_u^{R (e_i+e_j)}= \sqrt{2}\,\Lambda_u^{R (e_i+e_j)/\sqrt{2}}$.
\end{proof}

\begin{proof}[Proof of Proposition~\ref{prop:mainp}] We argue as in the proof of Theorem~\ref{thm:main}, with the measures $\mu\xz$ replaced by the measures
\begin{equation*}
(\mu^p)\xz:=|(\uxz)'|^p d \mathcal{L}^1 + \Ha^0\mres J_{\uxz},
\end{equation*}
thus evaluating
\begin{multline*}
     \e^{d-1}  \int_{Q} \sum_{\xi\in V}\sum_{i\in\e\Z^d}\big(|\xi\cdot(u(\e y + i + \e\xi)-u(\e y+ i))|^p\wedge 1 \big) dy
      \le \sum_{\xi\in V}\int_{\xi^\perp}\int_{\Omega\xz} d\mu^p_{\xi, z }d\Ha^{d-1}(z) \le \Lambda_u^{p,V}.
   \end{multline*}
   Therefore the proof proceeds as done for Theorem~\ref{thm:main}, first choosing suitably $y^\e$ from the previous inequality and Proposition~\ref{prop:disc} in Appendix~\ref{sec:disc}, and then replacing $|\xi\cdot (u(\e y^\e + i + \e\xi)-u(\e y^\e + i ))|$  by $|\xi\cdot (u(\e y^\e + i + \e\xi)-u(\e y^\e + i) )|^p$, to define the approximating functions $u^\e$ converging in measure to $u$. We still control $\mathcal{B}^\e$, the union of cubes of the second type, as in \eqref{controllobadcubes}; on any $\overline{Q}_\e=\e y^\e+i+\e [0,1)^d$ 
   of the first type, we have that
    \begin{equation*}
    \int_{\overline{Q}_\e} |e(u^\e)|^p dx\leq C \e^{d-1} \sum_{\xi\in V}|\xi\cdot(u(\e y^\e + i + \e\xi)-u(\e y^\e+ i))|^p 
    \end{equation*}
    by Proposition~\ref{prop:control} for $p>1$ and a change of variables. Summing up we obtain the bound \eqref{boundunifappp}.
\end{proof} 

 \begin{proof}[Proof of Corollary~\ref{cor:GSBDp}] For any $u_k$ let $(u^\e_k)_{\e>0}$ be the approximating family provided by Proposition~\ref{prop:mainp}.
Then for any $k \in \N \setminus \{0\}$ we can find $\e_k>0$ such that, for $\tilde{u}_k:=u^{\e_k}_k$ it holds that $\tilde{u}_k \in SBV^\infty(\Omega; \R^d)$  with $J_{\tilde{u}_k} \subset \partial^* \mathcal{B}^{\e_k}$ and $\mathcal{B}^{\e_k}$ union of finitely many cubes of sidelength $\e_k$ included in $\Omega$, such that
\begin{equation}\label{0202252211'}
\begin{split}
\|\arctan(\tilde{u}_k-u_k)\|_{L^1}&\leq \frac{1}{k},\\
\int_{ \Omega_{\e_k} } | e(\tilde{u}_k)|^p dx +  \Ha^{d-1}(\partial^* \mathcal{B}^{\e_k} \cap \Omega_{\e_k} ) 
     &\le C\,  \Lambda^{p, V_k}_{u_k}.
\end{split}
\end{equation}
Then $(\tilde{u}_k)_k$ satisfies the assumptions of \cite[Theorem~1.1]{ChaCri20AnnSNS} 
on any $\Omega_{\e_{\overline k}}$ for $k\geq \overline{k}$,  and therefore by that compactness result and by a diagonal argument on $\Omega_{\e_k}$, up to a (not relabelled) subsequence, there are
a Caccioppoli partition $\mathcal{P}=(P_n)_n$ of $\Omega$, a sequence of piecewise rigid motions $(a_k)_k$ satisfying \eqref{eqs:0808222034}, and
$u\in GSBD^p(\Omega)$ such that 
\begin{subequations}\label{eqs:0203200917''}
\begin{equation}\label{2202201910'''}
\tilde{u}_k-a_k \to u \quad \text{a.e.\ in }\Omega. 
\end{equation}
\end{subequations}
By the first in \eqref{0202252211'} we deduce \eqref{2202201910'}.
Eventually, we get \eqref{eq:convGradSym} and \eqref{eq:sciSalto} by \cite[Theorem~11.3]{DM13} (recall \cite[Lemma~2.1]{Fri19}) applied to $(u_k-a_k)_k$.
\end{proof}

%
%
 
\appendix
\section{Discrete energy estimate}\label{sec:energy}
In what follows, $(e_i)_{i=1}^d$ denotes the canonical basis
of $\R^d$.
\begin{lemma}\label{lem:rigid} Consider the unit cube $Q=[0,1]^d\subset\R^d$.
  Let $v\in (\R^d)^{\{0,1\}^d}$ be given at all vertices of $Q$ such
  that $v_i(x+e_i)=v_i(x)$ for any $x\in \{0,1\}^d$ with $x_i=0$
  and $v_i(x+e_i+e_j)+v_j(x+e_i+e_j)=v_i(x)+v_j(x)$ for any
  $x\in \{0,1\}^d$ with $x_i=x_j=0$. For $x\in Q$, we also denote by $v(x)$ the multilinear
  interpolation of the values $v$ at the vertices
  (affine on each $[x,x+e_i]$ for any $x\in Q$
  with $x_i=0$). Then $e(v)=0$ in $Q$ (so that, in fact, $v$ is affine with
  skew-symmetric gradient).
\end{lemma}
\begin{proof}
  We first assume $d=2$. Then,
  \[
    v(x) = (1-x_1)(1-x_2)v(0,0)+(1-x_1)x_2 v(0,1)+
    x_1(1-x_2)v(1,0)+x_1x_2 v(1,1).
  \]
  One easily sees that $\partial_1 v_1 = \partial_2 v_2=0$ everywhere.
  One has
  \begin{multline*}
    \partial_1v_2(x)+\partial_2 v_1(x)
    = (1-x_2)(v_2(1,0)-v_2(0,0))+x_2(v_2(1,1)-v_2(0,1))
    \\+ (1-x_1)(v_1(0,1)-v_1(0,0))+x_1(v_1(1,1)-v_1(1,0))
  \end{multline*}
  Since by assumption $v_1(0,0)=v_1(1,0)$, $v_1(0,1)=v_1(1,1)$,  $v_2(0,0)=v_2(0,1)$, $v_2(1,0)=v_2(1,1)$, this is also:
  \begin{multline*}
    \partial_1v_2(x)+\partial_2 v_1(x)
    = (1-x_2)(v_2(1,1)-v_2(0,0))+x_2(v_2(1,1)-v_2(0,0))
    \\+ (1-x_1)(v_1(1,1)-v_1(0,0))+x_1(v_1(1,1)-v_1(0,0))
    \\ = v_2(1,1)-v_2(0,0)+v_1(1,1)-v_1(0,0) = 0,
  \end{multline*}
  where the last equality follows again by assumption.
  Hence $e(v)=0$.
  
  Now, consider $d\ge 3$: then for any $\{i,j\}\subset \{1,\dots,d\}$,
  the case $d=2$ shows that $e(v)_{i,j}(x)=0$ for any $x$ with $(x_i,x_j)\in [0,1]^2$
  and $x_k\in \{0,1\}$, $k\not\in \{i,j\}$. Since at any other $x\in Q$,
  $e(v)_{i,j}(x)$ is a convex interpolation of those values, we find that $e(v)_{i,j}=0$ everywhere.
  Similarly, $e(v)_{i,i}(x) = \partial_i v_i(x)=0$ for any $x\in Q$ and $i\in\{1,\dots,d\}$.
  Hence $e(v)=0$.
\end{proof}
\begin{prop}\label{prop:control}
  Consider the unit cube $Q$, $v\in (\R^d)^{\{0,1\}^d}$ given at all vertices of $Q$,
  and the $d$-linear interpolation of $v$ inside $Q$. Then for every $p\in [1,+\infty)$  there is a constant $C>0$ depending on $d$ and $p$
  such that:
  \begin{multline}\label{eq:control}
    \int_Q |e(v)|^pdx \le C \Bigg( \sum_{i=1}^d
    \sum_{\begin{subarray}{c}x\in\{0,1\}^d\\ x_i=0\end{subarray}} |v_i(x+e_i)-v_i(x)|^p
    \\+ \sum_{i,j=1}^d \sum_{\begin{subarray}{c}x\in\{0,1\}^d\\ x_i=x_j=0\end{subarray}}
    |v_i(x+e_i+e_j)+v_j(x+e_i+e_j)-v_i(x)-v_j(x)|^p\Bigg)
  \end{multline}
\end{prop}
\begin{proof}
  If not, there is a sequence $(v_n)_{n\ge 1}$ which satisfies the reverse inequality,
  with $C$ replaced with $n$, arbitarily large. Renormalizing, we may assume that $\|e(v_n)\|_{L^p(Q)}=1$ and thus that
   the right-hand side of~\eqref{eq:control}, evaluated for $v_n$, goes
  to zero as $n\to\infty$.
  By Poincar\'e-Korn inequality, there is a sequence rigid motions (affine functions $a_n(x)=A_n x + b_n$ with
  skew-symmetric gradient $A_n$) such that:
  \[
    \int_Q |v_n(x)-a_n(x)|^p dx \le c\int_Q |e(v_n)|^p dx = c
  \]
  where $c$ depends only on the dimension  and on $p$.  Letting $v'_n= v_n-a_n$, it is bounded
  in $L^p(Q)$. Since it is, in fact, finite-dimensional (it is a polynomial of degree
  at most $d$), it is relatively compact and, up to a subsequence, converges to a limit $v'$
  with $\|e(v')\|_{L^p(Q)} = 1$. Yet the right-hand side of~\eqref{eq:control}, for $v'$,
  vanishes, so that $v'$ satisfies the
  assumptions of Lemma~\ref{lem:rigid} and we deduce $e(v')=0$. This is a contradiction,
  hence the corollary is true.
  An alternative proof consists in arguing
    that both left-hand and right-hand sides of~\eqref{eq:control} define
    $p$th-powers of norms on the finite dimensional space $(\R^d)^{\{0,1\}^d}/\sim$
    where $u\sim v$ if and only if the $d$-linear interpolation of $u-v$
    is an infinitesimal rigid motion.
\end{proof}
\section{Convergence of Discretizations}\label{sec:disc}

\newcommand{\Qq}{\mathcal{Q}}
Let $\Qq(x) = \chi_{[-1/2,1/2)^d}(x)$, $\Delta(x)=
\prod_{i=1}^d (1-|x_i|)^+$, for $x=(x_1,\dots,x_d)\in \R^d$, and for every $k \in \mathbb{N}$ let $T_k\colon \R^m\to \R^m$ be defined by
\begin{equation*}
T_k(t) = t \left(1 \wedge \frac{k}{|t|}\right).
\end{equation*}
We remark that for $u\in L^0(\R^d; \R^m)$
it holds that for any $R>0$,  
\begin{multline}\label{hppartenza}
\lim_{\ell\to\infty} \big|\{x \in B_R(0): |u(x)|>\ell \}\big| = \Big|\bigcap_{\ell>0} \{x\in B_R(0) :|u(x)|>\ell\}\Big| \\ = \big|\{x\in B_R(0) :|u(x)|=+\infty\}\big|=0.
\end{multline}
 Moreover, for any $\e>0$ and $y \in Q$ we denote by 
\[
\mathcal{D}^{\e}_y:=\{\e y+ z\colon z\in\e \Z^d\}, 
\]
the discretization of $\R^d$ formed by cubes (with the same orientation) of sidelength $\e$ and anchor point $y$ and by $u^{\e}_{y}\colon \R^d \to \R^m$ the corresponding discretized function defined by
\begin{equation}\label{uepsdiscr}
u^{\e}_y(x):=\sum_{\ol x \in \mathcal{D}^{\e}_y} u(\ol x) \Delta(\tfrac{x-\ol x}{\e}) \quad \text{for all }x\in \R^d.
\end{equation}  
In the following notation for the norms, we denote by $L^1(B_k)$ the space $L^1(B_k(0); \R^m)$, with $B_k \equiv B_k(0)\subset \R^d$ the ball of center 0 and radius $k$. 
\begin{lemma}\label{le:convtronc}
Let $u\in L^0(\R^d; \R^m)$
and assume that for all $\e>0$ there exists $y_\e \in Q$ such that
for any $k \in \N \setminus \{0\}$
\begin{equation}\label{limQ}
\lim_{\e\to 0} \Big\|\sum_{\ol x \in \mathcal{D}^{\e}_{y_\e}} T_k(u(\ol x))\Qq(\tfrac{\cdot-\ol x}{\e}) - T_k(u)\Big\|_{L^1(B_k)} \hspace{-0.7em}+ \Big\|\sum_{\ol x \in \mathcal{D}^{\e}_{y_\e}} T_k(u(\ol x))\Delta(\tfrac{\cdot -\ol x}{\e}) - T_k(u)\Big\|_{L^1(B_k)}=0.
\end{equation} 
Then 
\begin{equation}\label{eq:CVL1}
\lim_{\e\to 0} \|T_k(u^\e_{y_\e})-T_k(u)\|_{L^1(B_k)}=0 \quad\text{for any }k \in \N\setminus \{0\}.
\end{equation}
\end{lemma}

\begin{proof}
Let us fix 
$k\in \N \setminus \{0\}$. Observe that if $\ell>k$, since $T_k = T_k\circ T_\ell$
and $T_k$ is $1$-Lipschitz,
\begin{equation}\label{eq:CVL01}
  \Big\|  T_k \Big(\sum_{\ol x} T_{\ell}(u(\ol x))\Delta(\tfrac{\cdot -\ol x}{\e})\Big)
    - T_k (u)
  \Big\|_{L^1(B_k)}
  \le
  \Big\|  \sum_{\ol x} T_{\ell}(u(\ol x))\Delta(\tfrac{\cdot-\ol x}{\e}) - T_\ell (u)
  \Big\|_{L^1(B_k)}\to 0
\end{equation}
as $\e\to 0$, 
thanks to \eqref{limQ}. This means we just need to bound
\[
\Big\|  T_k(u^\e_{y_\e})-
 T_k \Big(\sum_{\ol x} T_{\ell}(u(\ol x))\Delta(\tfrac{\cdot -\ol x}{\e})\Big)\Big\|_{L^1(B_k)}.
\]
Observe that this quantity is zero, except possibly in the cubes with
center a point $\ol x$ such that $|u(\ol x)|>\ell$, in which case it is
bounded (in norm) by $2k$. Hence we get the bound
\begin{equation}\label{eq:CVL02}
\Big\|  T_k(u^\e_{y_\e})-
 T_k \Big(\sum_{\ol x} T_{\ell}(u(\ol x))\Delta(\tfrac{\cdot-\ol x}{\e})\Big) \Big\|_{L^1(B_k)}
\hspace{-.5em}\le  2k (2\e)^d \# \{ \ol x \in \mathcal{D}^\e_{y_\e}\cap B_{k+\e\sqrt{d}} \colon |u(\ol x)|>\ell\}.
\end{equation}
For any $x\in \ol x+[-\e/2,\e/2)^d$ with $|u(\ol x)|>\ell$,  either
$|u(x)|>\ell/2$, or
\[  |T_\ell(u(\ol x))-T_\ell(u(x))|
  \ge |T_\ell(u(\ol x))|-|T_\ell(u(x))|=\ell-|u(x)|\ge \frac{\ell}{2}.
\]
Hence (for $\ell$ large enough),
\begin{multline*}
  \e^d \# \{ \ol x \in \mathcal{D}^\e_{y_\e} \cap B_{k+\e\sqrt{d}}\colon |u(\ol x)|>\ell\}
  \\
 \le \Big| \Big\{x \in B_{k+\e\sqrt{d}}: |u(x)|> \tfrac{\ell}{2} \Big\} \Big| +
\Big| \Big\{x\in B_\ell  :\Big|\sum_{\ol x} T_\ell(u(\ol x)) \Qq(\tfrac{x-\ol x}{\e})
- T_\ell(u(x))\Big|\ge \tfrac{\ell}{2}\Big\}\Big|
\\ \le
\Big| \Big\{x \in B_{k+\e\sqrt{d}}: |u(x)|> \tfrac{\ell}{2} \Big\} \Big| +
\frac{2}{\ell} \int_{B_\ell} \Big|\sum_{\ol x} T_\ell(u(\ol x)) \Qq(\tfrac{x-\ol x}{\e})
- T_\ell(u(x))\Big| dx.
\end{multline*}
We deduce from~\eqref{limQ}, \eqref{eq:CVL01}, and \eqref{eq:CVL02} that
for any $\ell>k$,
\[
  \limsup_{\e\to 0}
  \left\|T_k(u^\e_{y_\e})-T_k(u)\right\|_{L^1(B_k)}
  \le 2^{d+1} k \Big| \Big\{x \in B_{k+\e\sqrt{d}} : |u(x)|> \tfrac{\ell}{2} \Big\} \Big|.
\]
Sending $\ell\to\infty$,  by \eqref{hppartenza}  we obtain~\eqref{eq:CVL1}.
\end{proof}
Recalling the notation \eqref{uepsdiscr}, the following lemma holds.
\begin{lemma}\label{le:convae}
Let $u\in L^0(\R^d; \R^m)$ 
and $(y_\e)_{\e>0}$ be such that \eqref{eq:CVL1} holds. 
Then 
 \begin{equation}\label{convae}
u^\e_{y_\e} \to u \quad  \text{locally in } {L^0(\R^d; \R^m). }
\end{equation} 
\end{lemma}
\begin{proof}
    Fix $\eta>0$, $R>0$. For any $x$ and any $k\in \mathbb{N}\setminus \{0\}$, $k>2\eta$,
    it holds
    that if $|u^\e_{y_\e}(x)-u(x)|>\eta$ then either $|u(x)|>k/2$, or
    $|T_k(u^\e_{y_\e}(x)) - T_k(u(x))|>\eta$. Hence, for $k\ge R$,
    \begin{equation*}
      \big|\{x\in B_R(0): |u^\e_{y_\e}(x)-u(x)|>\eta\}\big|
      \le \big|\{x\in B_R(0):|u(x)|>\tfrac{k}{2} \}\big|
      +\frac{1}{\eta}\| T_k(u^\e_{y_\e}) - T_k(u) \|_{L^1(B_k)}.
    \end{equation*}
    Hence, for any fixed $k$,
    \[
      \limsup_{\e\to 0} \big|\{x\in B_R(0): |u^\e_{y_\e}(x)-u(x)|>\eta\}\big|\le \big|\{x\in B_R(0):|u(x)|>\tfrac{k}{2} \}\big|.
    \]
    The conclusion follows letting $k\to\infty$.
\end{proof}
%

\begin{prop}\label{prop:disc}
   For any $u\in L^0(\R^d; \R^m)$
   and for all $\e>0$ there exists 
  a measurable set $Q_\e\subset Q=[0,1)^d$ 
  such that both
  $\lim_{\e\to 0} |Q_\e|=1$ and $u^\e_{y_\e} \to u$ locally in $L^0(\R^d; \R^m)$
   if 
  $y_\e\in Q_\e$ for every $\e>0$.
\end{prop}
\begin{proof}
 By Lemma~\ref{le:convtronc} and Lemma~\ref{le:convae}, it is enough to find a set $Q_\e$ such that
 \eqref{limQ} holds for any choice of $y_\e\in Q_\e$.
  For $y\in Q$, we let:
  \begin{multline*}
    \Phi_\e(y) = \sum_{k\ge 1} 2^{-k}
    \int_{ B_k } \Big|\sum_{\ol x \in \mathcal{D}^{\e}_{y}} T_k(u(\ol x))\Qq(\tfrac{\cdot-\ol x}{\e}) - T_k(u)\Big|
     + \Big|\sum_{\ol x \in \mathcal{D}^{\e}_{y}} T_k(u(\ol x))\Delta(\tfrac{\cdot-\ol x}{\e}) - T_k(u)\Big| dx.
  \end{multline*}
  For any bounded function 
  with compact support
  $v\in L^\infty(\R^d)$, and
  a continuous function $\psi:\R^d\to [0,1]$ with compact support
  such that $\sum_{z\in\Z^d}\psi(x-z)=1$ for any $x\in\R^d$,
  one has
  \begin{multline*}
    \int_Q \int \Big|\sum_{\ol x \in \mathcal{D}^{\e}_{y}} v(\ol x)\psi(\tfrac{x-\ol x}{\e}) - v(x)\Big|dx \,dy
    =
    \int \Big| \int_Q \sum_{\ol x \in \mathcal{D}^{\e}_{y}} (v(\ol x)-v(x))\psi(\tfrac{x-\ol x}{\e}) dy \Big|dx
    \\
    \stackrel{y'=\e y}{\le}
    \int  \int_{\e Q} \sum_{z\in \e\Z^d} |v(y'+z)-v(x)|\psi(\tfrac{x-z-y'}{\e}) \e^{-d}dy'\, dx
    = \int \int \e^{-d}\psi(\tfrac{x-y'}{\e}) |v(y')-v(x)|dy' dx
    \\ \stackrel{x-y'=\e\xi}{=}
    \int \psi(\xi) \|v(\cdot-\e\xi)-v\|_{L^1} d\xi
  \end{multline*}
  where we have set $\psi_\e:= \psi(\cdot/\e)\e^{-d}$.

  Hence,  for $v=T_k(u) \chi_{B_k}$ and $\psi=\Qq,\, \Delta$, 
  \[
    \int_Q \Phi_\e(y) dy
    \le \sum_{k\ge 1} 2^{-k} \int (\Qq(\xi)+\Delta(\xi))\|T_k(u)(\cdot-\e\xi)-T_k(u)\|_{L^1( B_k )}
  \]
  Since the terms in the sum above are uniformly bounded 
 (by $4 k2^{-k}(1+| B_k |)$ 
  for $\e$ small enough) and go to zero, we deduce that
  \[
  \lim_{\e\to 0} \int_Q\Phi_\e(y)dy=0.
  \]
  We can then choose $Q_\e := \Big\{ y\in Q: \Phi_\e(y) \le \sqrt{\int_Q \Phi_\e}\Big\}$,
  which is such that $|Q\setminus Q_\e|\le  \sqrt{\int_Q \Phi_\e}$.
  If one chooses $y_\e\in Q_\e$, then for each $k$,
  both $L^1$ norms in \eqref{limQ} are bounded by $2^k \sqrt{\int_Q \Phi_\e}$, which goes to zero.  Then we conclude by Lemma~\ref{le:convae}.
\end{proof}

\begin{remark} 
Proposition~\ref{prop:disc} implies that $u^\e_{y_\e}$ converge to $u$ a.e.\ in $\R^d$, up to a subsequence $\e_n \downarrow 0$. We remark that such convergence may be recovered directly, without using Lemma~\ref{le:convae}, by modifying the argument in the proof of Proposition~\ref{prop:disc} as follows.  
  With the notation in the proof of Proposition~\ref{prop:disc}, first let $\e_n$ such that $\sqrt{\int_Q\Phi_{\e_n}}\le 2^{-n}$.
  Then, letting $Q_n := \bigcap_{h\ge n} Q_{\e_h}$, we find that $|Q\setminus Q_n|\le 2^{-n+1}$
  and for any $y\in Q_n$ and $h\ge n$, $\Phi_{\e_h}(y)\le 2^{-h}$.
  Therefore, for $\tilde{Q}:=\bigcup_{n\ge 1} Q_n$, we find that $|\tilde{Q}|=1$ and for any $y\in \tilde{Q}$, 
  $\Phi_{\e_h}(y)\le 2^{-h}$ for $h$ large enough (depending on $y$).

  Next, once the main result is proven, we can choose a (further) subsequence
  $\e_k\downarrow 0$ such that for each $k$,
\[
    \|T_k(u^{\e_k}_{y_{\e_k}})-T_k(u)\|_{L^1( B_k )}\le 2^{-k}.    
  \]
  Then, $\sum_k |\big(T_k(u^{\e_k}_{y_{\e_k}}(x))-T_k(u(x))\big) \chi_{B_k}(x) |<+\infty$
  for a.e.~$x\in\R^d$, which shows that
  \[
    u^{\e_k}_{y_{\e_k}} \to u \quad \text{a.e.\ in }\R^d \text{ as }k\to +\infty.
  \]
 As a consequence, if one considers instead a \textit{subsequence}
    in Proposition~\ref{prop:disc}, one can replace $Q_\e$ by the
    set $\tilde{Q}$ of full measure, independent on $\e$,
    and assume that the convergence is almost everywhere.
\end{remark}
\bigskip
\noindent {\bf Acknowledgements.} 
V.C.~acknowledges the support of the MUR - PRIN project 2022J4FYNJ CUP B53D23009320006 ``Variational methods for stationary and evolution problems with singularities and interfaces'', PNRR Italia Domani, funded by the European Union via the program NextGenerationEU. V.C. is member of the Gruppo Nazionale per l'Analisi Matematica, la Probabilità e le loro Applicazioni (INdAM-GNAMPA).
Most of this work was done while A.C.~was visiting the Dipartimento di Matematica,
Universit\`a di Roma--La Sapienza, thanks to the ``Dipartimento di Eccellenza Project CUP B83C23001390001''.

\end{document}